\documentclass[11pt,reqno]{amsart}

\usepackage{url}
\usepackage{amssymb}
\usepackage{amscd}
\usepackage{amsfonts}
\usepackage{amsmath}
\usepackage{amsthm}
\usepackage{amsgen}

\newcommand{\sgp}{semi\-group}
\newcommand{\sgps}{semi\-groups}
\newcommand{\fs}{finite \sgp}
\newcommand{\fss}{finite \sgps}
\newcommand{\fb}{finitely based}

\newcommand{\nfb}{non\-finitely based}
\newcommand{\fg}{finitely generated}
\newcommand{\is}{involution semi\-group}

\newcommand{\fbp}{finite basis problem}

\def\FI{\ensuremath{\mathcal{FI}}}
\def\malcev{\mathop{\hbox{$\bigcirc$\kern-9.5pt\raise1pt\hbox{\scriptsize$m$}\kern1.5pt}}}

\DeclareMathOperator{\var}{\mathsf{var}} \DeclareMathOperator{\Diag}{\mathsf{Diag}}

\theoremstyle{plain}
\newtheorem{theorem}{Theorem}
\newtheorem{lemma}[theorem]{Lemma}
\newtheorem{proposition}[theorem]{Proposition}

\theoremstyle{remark}
\newtheorem*{question}{Question}

\title[A nonfinitely based semigroup of triangular matrices]{A nonfinitely based semigroup\\
of triangular matrices}

\author{M. V. Volkov}
\address{Institute of Mathematics and Computer Science, Ural Federal University,
Lenina 51, 620000 Ekaterinburg, Russia} \email{mikhail.volkov@usu.ru}

\begin{document}

\begin{abstract}
A new sufficient condition under which a semigroup admits no finite identity basis has been recently suggested in a joint paper by Karl
Auinger, Yuzhu Chen, Xun Hu, Yanfeng Luo, and the author. Here we apply this condition to show the absence of a finite identity basis for
the semigroup $\mathrm{UT}_3(\mathbb{R})$ of all upper triangular real $3\times 3$-matrices with 0s and/or 1s on the main diagonal. The
result holds also for the case when $\mathrm{UT}_3(\mathbb{R})$ is considered as an involution semigroup under the reflection with respect
to the secondary diagonal.
\end{abstract}

\maketitle

\section*{Introduction}

A \emph{\sgp\ identity} is just a pair of \emph{words}, i.e., elements of the free semigroup $A^+$ over an alphabet $A$. In this paper
identities are written as ``bumped'' equalities such as $u\bumpeq v$. The identity $u\bumpeq v$ is \emph{trivial} if $u=v$ and
\emph{nontrivial} otherwise. A \sgp\ $S$ \emph{satisfies} $u\bumpeq v$ where $u,v\in A^+$ if for every homomorphism $\varphi:A^+\to S$, the
equality $u\varphi=v\varphi$ is valid in $S$; alternatively, we say that $u\bumpeq v$ \emph{holds} in $S$. Clearly, trivial identities hold
in every semigroup, and there exist semigroups (for instance, free semigroups over non-singleton alphabets) that satisfy only trivial
identities.

Given any system $\Sigma$ of \sgp\ identities, we say that an identity $u\bumpeq v$ \emph{follows} from $\Sigma$ if every \sgp\ satisfying
all identities of $\Sigma$ satisfies the identity $u\bumpeq v$ as well; alternatively, we say that $\Sigma$ \emph{implies} $u\bumpeq v$. A
\sgp\ $S$ is said to be \emph{\fb} if there exists a finite identity system $\Sigma$ such that every identity holding in $S$ follows from
$\Sigma$; otherwise $S$ is called \emph{\nfb}.

The \fbp, that is, the problem of classifying \sgps\ according to the finite basability of their identities, has been intensively explored
since the mid-1960s when the very first examples of \nfb\ \sgps\ were discovered by Austin \cite{Au66}, Biryukov~\cite{Bi65}, and Perkins
\cite{Pe66,Pe69}. One of the examples by Perkins was especially impressive as it involved a very transparent and natural object. Namely,
Perkins proved that the finite basis property fails for the 6-element \sgp\ formed by the following integer $2\times 2$-matrices under the
usual matrix multiplication:
$$\begin{pmatrix} 0 & 0\\ 0 & 0\end{pmatrix},\
\begin{pmatrix} 1 & 0\\ 0 & 0\end{pmatrix},\
\begin{pmatrix} 0 & 1\\ 0 & 0\end{pmatrix},\
\begin{pmatrix} 0 & 0\\ 1 & 0\end{pmatrix},\
\begin{pmatrix} 0 & 0\\ 0 & 1\end{pmatrix},\
\begin{pmatrix} 1 & 0\\ 0 & 1\end{pmatrix}.$$
Thus, even a \fs\ can be \nfb; moreover, it turns out that \sgps\ are the only ``classical'' algebras for which finite \nfb\ objects can
exist: finite groups \cite{OaPo64}, finite associative and Lie rings \cite{Kr73,Lv73,BaOl75}, finite lattices \cite{McK70} are all \fb.
Therefore studying \fss\ from the viewpoint of the \fbp\ has become a hot area in which many neat results have been achieved and several
powerful methods have been developed, see the survey \cite{Vo01} for an overview.

It may appear surprising but the \fbp\ for \textbf{infinite} \sgps\ is less studied. The reason for this is that infinite semigroups
usually arise in mathematics as \sgps\ of transformations of an infinite set, or \sgps\ of relations on an infinite domain, or \sgps\
matrices over an infinite ring, and as a rule all these \sgps\ are ``too big'' to satisfy any nontrivial identity. For instance (see, e.g.,
\cite{BCK04}), the two integer matrices
$$\begin{pmatrix} 2 & 0\\ & 1\end{pmatrix},\
\begin{pmatrix} 2 & 1\\   & 1\end{pmatrix}$$
are known to generate a free subsemigroup in the \sgp\ $\mathrm{T}_2(\mathbb{Z})$ of all upper triangular integer $2\times 2$-matrices.
(Here and below we omit zero entries under the main diagonal when dealing with upper triangular matrices.) Therefore even such a ``small''
matrix \sgp\ as $\mathrm{T}_2(\mathbb{Z})$ satisfies only trivial identities, to say nothing about matrix \sgps\ of larger dimension.

If all identities holding in a \sgp\ $S$ are trivial, $S$ is \fb\ in a void way, so to speak. If, however, an infinite semigroup satisfies
a nontrivial identity, its \fbp\ may constitute a challenge since ``finite'' methods are non-applicable in general. Therefore, up to
recently, results classifying \fb\ and \nfb\ members within natural families of concrete infinite semigroups that contain semigroups with a
nontrivial identity have been rather sparse.

Auinger, Chen, Hu, Luo, and the author~\cite{ACHLV} have found a new sufficient condition under which a semigroup is \nfb\ and applied this
condition to certain important classes of infinite \sgps. In the present paper we demonstrate yet another application; its interesting
feature is that it requires the full strength of the main result of~\cite{ACHLV}. Namely, we prove that the \sgp\
$\mathrm{UT}_3(\mathbb{R})$ of all upper triangular real $3\times 3$-matrices whose main diagonal entries are 0s and/or 1s is \nfb. The
result holds also for the case when $\mathrm{UT}_3(\mathbb{R})$ is considered as an involution semigroup under the reflection with respect
to the secondary diagonal.

The paper is structured as follows. In Section~\ref{sec:nfb} we recall the main result from~\cite{ACHLV}, and in
Section~\ref{sec:application} we apply it to the semigroup $\mathrm{UT}_3(\mathbb{R})$.  Section~\ref{sec:open problems} collects some
concluding remarks and a related open question.

An effort has been made to keep this paper self-contained, to a reasonable extent. We use only the most basic concepts of semigroup theory
and universal algebra that all can be found in the early chapters of the textbooks \cite{BS81,CP61}, a suitable version of the main theorem
from~\cite{ACHLV}, and a few minor results from~\cite{ADV12,Ma53,Sa87}.

\section{A sufficient condition for the non-existence of a finite basis}
\label{sec:nfb}

The sufficient condition for the non-existence of a finite basis established in~\cite{ACHLV} applies to both plain semigroups, i.e., \sgps\
treated as algebras of type (2), and semigroups with involution as algebras of type (2,1). Let us recall all the concepts needed to
formulate this condition.

We start with the definition of an \is. An algebra $\langle S,\cdot,{}^\star\rangle$ of type (2,1) is called an \emph{\is} if $\langle
S,\cdot\rangle$ is a semigroup (referred to as the \emph{semigroup reduct} of $\langle S,\cdot,{}^\star\rangle$) and the unary operation
$x\mapsto x^\star$ is an involutory anti-automorphism of $\langle S,\cdot\rangle$, that is,
$$(xy)^\star=y^\star x^\star \text{ and  } (x^\star)^\star=x$$
for all $x,y\in S$.

The \emph{free \is} $\FI(A)$ on a given alphabet $A$ can be constructed as follows. Let $\overline{A}:=\{a^\star\mid a\in A\}$ be a
disjoint copy of $A$. Define $(a^\star)^\star:=a$ for all $a^\star\in\overline{A}$. Then $\FI(A)$ is the free \sgp\ $(A\cup\overline{A})^+$
endowed with the involution defined by
$$(a_1\cdots a_m)^\star:= a_m^\star\cdots a_1^\star$$
for all $a_1,\dots,a_m\in A\cup \overline{A}$. We refer to elements of $\FI(A)$ as \emph{involutory words over $A$}. An \emph{involutory
identity} $u\bumpeq v$ is just a pair of involutory words; the identity \emph{holds} in an \is\ $S$ if for every \is\ homomorphism
$\varphi:\FI(A)^+\to S$, the equality $u\varphi=v\varphi$ is valid in $S$. Now the concepts of a \fb/\nfb\ \is\ are defined exactly as in
the plain \sgp\ case. In what follows, we use square brackets to indicate adjustments to be made in the involution case.

A class $\mathbf{V}$ of [involution] \sgps\ is called a \emph{variety} if there exits a system $\Sigma$ of [involution] \sgp\ identities
such that $\mathbf{V}$ consists precisely of all [involution] \sgps\ that satisfy every identity in $\Sigma$. In this case we say that
$\mathbf{V}$ is \emph{defined} by $\Sigma$. If the system $\Sigma$ can be chosen to be finite, the corresponding variety is said to be
\emph{\fb}; otherwise it is \emph{\nfb}. Given a class $\mathbf{K}$ of [involution] semigroups, the variety defined by the identities that
hold in each [involution] semigroup from $\mathbf{K}$ is said to be \emph{generated by $\mathbf{K}$} and is denoted by $\var\mathbf{K}$; if
$\mathbf{K}=\{{S}\}$, we write $\var{S}$ rather than $\var\{{S}\}$. It should be clear that $S$ and $\var S$ are simultaneously \fb\ or
\nfb.

A \sgp\ is said to be \emph{periodic} if each of its one-generated subsemigroups is finite and \emph{locally finite} if each of its \fg\
subsemigroups is finite. A variety of semigroups is \emph{locally finite} if all its members are locally finite.

Let $\mathbf{A}$ and $\mathbf{B}$ be two classes of \sgps. The \emph{Mal'cev product} $\mathbf{A}\malcev\mathbf{B}$ of $\mathbf{A}$ and
$\mathbf{B}$ is the class of all \sgps\ $S$ for which there exists a congruence $\theta$ such that the quotient \sgp\ $S/\theta$ lies in
$\mathbf{B}$ while all $\theta$-classes that are sub\sgps\ in $S$ belong to $\mathbf{A}$. Notice that for a congruence $\theta$ on a
semigroup $S$, a $\theta$-class forms a subsemigroup of $S$ if and only if the class is an idempotent of the quotient \sgp\ $S/\theta$.

Let $x_1,x_2,\dots,x_n,\dots$ be a sequence of letters. The sequence $\{Z_n\}_{n=1,2,\dots}$ of \emph{Zimin words} is defined inductively
by $Z_1:=x_1$, $Z_{n+1}:=Z_nx_{n+1}Z_n$. We say that a word $v$ is an [\emph{involutory}] \emph{isoterm} for a class $\mathbf{C}$ of
semigroups [with involution] if the only [involutory] word $v'$ such that all members of $\mathbf{C}$ satisfy the [involution] semigroup
identity $v\bumpeq v'$ is the word $v$ itself.

Now we are in a position to state the main result of \cite{ACHLV}. Here $\mathbf{Com}$ denotes the variety of all commutative semigroups.

\begin{theorem}[{\mdseries\cite[Theorem~6]{ACHLV}}]
\label{thm:main}
A variety $\mathbf{V}$ of $[$involution$]$ semigroups is nonfinitely based provided that
\begin{enumerate}
\item[(i)] $[$the class of all semigroup reducts of\,$]$ $\mathbf{V}$ is contained in the variety
$\var(\mathbf{Com}\malcev\mathbf{W})$ for some locally finite semigroup variety $\mathbf{W}$ and
\item[(ii)] each Zimin word is an $[$involutory$]$ isoterm relative to $\mathbf{V}$.
\end{enumerate}
\end{theorem}

Formulated as above, Theorem~\ref{thm:main} suffices for all applications presented in~\cite{ACHLV} but is insufficient for the purposes of
the present paper. However, it is observed in \cite[Remark 1]{ACHLV} that the theorem remains valid if one replaces the condition (i) by
the following weaker condition:
\begin{enumerate}
\item[(i$'$)] \emph{$[$the class of all semigroup reducts of\,$]$ $\mathbf{V}$ is contained in the variety
$\var(\mathbf{U}\malcev\mathbf{W})$ where $\mathbf{U}$ is a semigroup variety all of whose periodic members are locally finite and
$\mathbf{W}$ is a locally finite semigroup variety.}
\end{enumerate}
Here we will utilize this stronger form of Theorem~\ref{thm:main}.

\section{The identities of $\mathrm{UT}_3(\mathbb{R})$}
\label{sec:application}

Recall that we denote by $\mathrm{UT}_3(\mathbb{R})$ the \sgp\ of all upper triangular real $3\times 3$-matrices whose main diagonal
entries are 0s and/or 1s.  For each matrix $\alpha\in\mathrm{UT}_3(\mathbb{R})$, let $\alpha^D$ stand for the matrix obtained by reflecting
$\alpha$ with respect to the secondary diagonal (from the top right to the bottom left corner); in other words,
$(\alpha_{ij})^D:=(\alpha_{4-j\,4-i})$. Then it is easy to verify that the unary operation $\alpha\mapsto\alpha^D$ (called the \emph{skew
transposition}) is an involutory anti-automorphism of $\mathrm{UT}_3(\mathbb{R})$. Thus, we can consider $\mathrm{UT}_3(\mathbb{R})$ also
as an \is. Our main result is the following

\begin{theorem}
\label{thm:matrix} The \sgp\ $\mathrm{UT}_3(\mathbb{R})$ is \nfb\ as both a plain \sgp\ and an \is\ under the skew transposition.
\end{theorem}

\begin{proof}
We will verify that the [involution] \sgp\ variety $\var\mathrm{UT}_3(\mathbb{R})$ satisfies the conditions (i$'$) and (ii) discussed at
the end of Section~\ref{sec:nfb}; the desired result will then follow from Theorem~\ref{thm:main} in its stronger form.

Let $\mathrm{D}_3$ denote the 8-element sub\sgp\ of $\mathrm{UT}_3(\mathbb{R})$ consisting of all diagonal matrices. To every matrix
$\alpha\in\mathrm{UT}_3(\mathbb{R})$ we assign the diagonal matrix $\Diag(\alpha)\in\mathrm{D}_3$ by changing each non-diagonal entry of
$\alpha$ to~0. The following observation is obvious.

\begin{lemma}
\label{lm:support} The map $\alpha\mapsto\Diag(\alpha)$ is a homomorphism of $\mathrm{UT}_3(\mathbb{R})$ onto $\mathrm{D}_3$.
\end{lemma}

We denote by $\theta$ the kernel of the homomorphism of Lemma~\ref{lm:support}, i.e.,
$$(\alpha,\beta)\in\theta \text{ if and only if } \Diag(\alpha)=\Diag(\beta).$$
Then $\theta$ is a congruence on $\mathrm{UT}_3(\mathbb{R})$. Since each element of the \sgp\ $\mathrm{D}_3$ is an idempotent, each
$\theta$-class is a subsemigroup of $\mathrm{UT}_3(\mathbb{R})$. The next fact is the core of our proof.

\begin{proposition}
\label{pr:classes} Each $\theta$-class of $\mathrm{UT}_3(\mathbb{R})$ satisfies the identity
\begin{equation}
\label{eq:Z4} Z_4\bumpeq (x_1x_2)^2x_1x_3x_1x_4x_1x_3x_1(x_2x_1)^2.
\end{equation}
\end{proposition}

\begin{proof}
We have to consider 8 cases. First we observe that the identity~\eqref{eq:Z4} is left-right symmetric, and therefore, \eqref{eq:Z4} holds
in some sub\sgp\ $S$ of $\mathrm{UT}_3(\mathbb{R})$ if and only if it holds in the sub\sgp\ $S^D=\{s^D\mid s\in S\}$ since $S^D$ is
anti-isomorphic to $S$. This helps us to shorten the below analysis.

\smallskip

\emph{\textbf{Case 1:}} $S_{000}=\left\{\left(\begin{smallmatrix}
0 &  \alpha_{12} & \alpha_{13}\\
  & 0            & \alpha_{23}\\
  &              & 0
\end{smallmatrix}\right)\mid \alpha_{12},\alpha_{13},\alpha_{23}\in\mathbb{R}\right\}$. This sub\sgp\ is easily seen to satisfy the identity
$x_1x_2x_3\bumpeq y_1y_2y_3$ which clearly implies~\eqref{eq:Z4}.

\smallskip

\emph{\textbf{Case 2:}} $S_{100}=\left\{\left(\begin{smallmatrix}
1 &  \alpha_{12} & \alpha_{13}\\
  & 0            & \alpha_{23}\\
  &              & 0
\end{smallmatrix}\right)\mid \alpha_{12},\alpha_{13},\alpha_{23}\in\mathbb{R}\right\}$. Multiplying 3 arbitrary matrices
$\alpha,\beta,\gamma\in S_{100}$, we get
\begin{multline*}\begin{pmatrix}
1 & \alpha_{12} & \alpha_{13}\\
  & 0           & \alpha_{23}\\
  &             & 0
\end{pmatrix}\cdot
\begin{pmatrix}
1 & \beta_{12}  & \beta_{13}\\
  & 0           & \beta_{23}\\
  &             & 0
\end{pmatrix}\cdot
\begin{pmatrix}
1 & \gamma_{12} & \gamma_{13}\\
  & 0           & \gamma_{23}\\
  &             & 0
\end{pmatrix}=\\
\begin{pmatrix}
1 & \gamma_{12} & \gamma_{13}+\beta_{12}\gamma_{23}\\
  & 0           & 0\\
  &             & 0
\end{pmatrix}=\begin{pmatrix}
1 & \beta_{12}  & \beta_{13}\\
  & 0           & \beta_{23}\\
  &             & 0
\end{pmatrix}\cdot
\begin{pmatrix}
1 & \gamma_{12} & \gamma_{13}\\
  & 0           & \gamma_{23}\\
  &             & 0
\end{pmatrix}.
\end{multline*}
Thus, $\alpha\beta\gamma=\beta\gamma$ and we have proved that $S_{100}$ satisfies the identity $xyz\bumpeq yz$. Clearly, this identity
implies~\eqref{eq:Z4}.

\smallskip

\emph{\textbf{Case 3:}} $S_{010}=\left\{\left(\begin{smallmatrix}
0 &  \alpha_{12} & \alpha_{13}\\
  & 1            & \alpha_{23}\\
  &              & 0
\end{smallmatrix}\right)\mid \alpha_{12},\alpha_{13},\alpha_{23}\in\mathbb{R}\right\}$. It is easy to see that this sub\sgp\ satisfies
the identity $xyx\bumpeq x$ which clearly implies~\eqref{eq:Z4}.

\smallskip

\emph{\textbf{Case 4:}} $S_{001}=\left\{\left(\begin{smallmatrix}
0 &  \alpha_{12} & \alpha_{13}\\
  & 0            & \alpha_{23}\\
  &              & 1
\end{smallmatrix}\right)\mid \alpha_{12},\alpha_{13},\alpha_{23}\in\mathbb{R}\right\}$. This case reduces to Case 2 since
$S_{001}=S_{100}^D$.

\smallskip

\emph{\textbf{Case 5:}} $S_{110}=\left\{\left(\begin{smallmatrix}
1 &  \alpha_{12} & \alpha_{13}\\
  & 1            & \alpha_{23}\\
  &              & 0
\end{smallmatrix}\right)\mid \alpha_{12},\alpha_{13},\alpha_{23}\in\mathbb{R}\right\}$. Multiplying 3 arbitrary matrices $\alpha,\beta,\gamma\in S_{110}$,
we get
\begin{multline*}
\begin{pmatrix}
1 & \alpha_{12} & \alpha_{13}\\
  & 1           & \alpha_{23}\\
  &             & 0
\end{pmatrix}\cdot
\begin{pmatrix}
1 & \beta_{12}  & \beta_{13}\\
  & 1           & \beta_{23}\\
  &             & 0
\end{pmatrix}\cdot
\begin{pmatrix}
1 & \gamma_{12} & \gamma_{13}\\
  & 1           & \gamma_{23}\\
  &             & 0
\end{pmatrix}=\\
\begin{pmatrix}
1 & \alpha_{12}+\beta_{12}+\gamma_{12} & \gamma_{13}+(\alpha_{12}+\beta_{12})\gamma_{23}\\
  & 1           & \gamma_{23}\\
  &             & 0
\end{pmatrix}
\end{multline*}
whence the product $\alpha\beta\gamma$ depends only on $\gamma$ and on the sum $\alpha_{12}+\beta_{12}$. Thus,
$\alpha\beta\gamma=\beta\alpha\gamma$ and we have proved that $S_{110}$ satisfies the identity $xyz\bumpeq yxz$. This identity
implies~\eqref{eq:Z4}.

\smallskip

\emph{\textbf{Case 6:}} $S_{101}=\left\{\left(\begin{smallmatrix}
1 &  \alpha_{12} & \alpha_{13}\\
  & 0            & \alpha_{23}\\
  &              & 1
\end{smallmatrix}\right)\mid \alpha_{12},\alpha_{13},\alpha_{23}\in\mathbb{R}\right\}$. Take an arbitrary homomorphism
$\varphi:\{x_1,x_2,x_3,x_4\}^+\to S_{101}$ and let $\alpha=x_1\varphi$, $\beta=x_2\varphi$, $\gamma=x_3\varphi$, and $\delta=x_4\varphi$.
Then one can calculate that both $Z_4\varphi$ and $(x_1x_2x_1x_2x_1x_3x_1x_4x_1x_3x_1x_2x_1x_2x_1)\varphi$ are equal to the matrix
$\left(\begin{smallmatrix}
1 & \alpha_{12} & \varepsilon\\
  & 0           & \alpha_{23}\\
  &             & 1
\end{smallmatrix}\right)$ where $\varepsilon$ stands for the following expression:
$$8\alpha_{13}+4\beta_{13}+2\gamma_{13}+\delta_{13}+\alpha_{12}(4\beta_{23}+2\gamma_{23}+\delta_{23})
+(4\beta_{12}+2\gamma_{12}+\delta_{12})\alpha_{23}.$$ Thus, the identity~\eqref{eq:Z4} holds on $S_{101}$.

For readers familiar with the Rees matrix construction (cf. \cite[Chapter~3]{CP61}), we outline a more conceptual proof for the fact that
$S_{101}$ satisfies~\eqref{eq:Z4}. Let $G=\langle\mathbb{R},+\rangle$ stand for the additive group of real numbers and let $P$ be the
$\mathbb{R}{\times}\mathbb{R}$-matrix over $G$ whose element in the $r$th row and the $s$th column is equal to $rs$. One readily verifies
that the map $\left(\begin{smallmatrix}
1 & \alpha_{12} & \alpha_{13}\\
  & 0           & \alpha_{23}\\
  &             & 1
\end{smallmatrix}\right)\mapsto (\alpha_{23},\alpha_{13},\alpha_{12})$
constitutes an isomorphism of the \sgp\ $S_{101}$ onto the Rees matrix \sgp\ $\mathrm{M}(\mathbb{R},G,\mathbb{R};P)$. It is known (see,
e.g., \cite{KR79}) and easy to verify that every Rees matrix \sgp\ over an Abelian group satisfies each identity $u\bumpeq v$ for which the
following three conditions hold: the first letter of $u$ is the same as the first letter of $v$;  the last letter of $u$ is the same as the
last letter of $v$; for each ordered pair of letters the number of occurrences of this pair is the same in $u$ and $v$. Inspecting the
identity~\eqref{eq:Z4}, one immediately sees that it satisfies the three conditions whence it holds in the \sgp\
$\mathrm{M}(\mathbb{R},G,\mathbb{R};P)$ and also in the \sgp\ $S_{101}$ isomorphic to $\mathrm{M}(\mathbb{R},G,\mathbb{R};P)$.

\smallskip

\emph{\textbf{Case 7:}} $S_{011}=\left\{\left(\begin{smallmatrix}
0 &  \alpha_{12} & \alpha_{13}\\
  & 1            & \alpha_{23}\\
  &              & 1
\end{smallmatrix}\right)\mid \alpha_{12},\alpha_{13},\alpha_{23}\in\mathbb{R}\right\}$. This case reduces to Case 5 since
$S_{011}=S_{110}^D$.

\smallskip

\emph{\textbf{Case 8:}} $S_{111}=\left\{\left(\begin{smallmatrix}
1 &  \alpha_{12} & \alpha_{13}\\
  & 1            & \alpha_{23}\\
  &              & 1
\end{smallmatrix}\right)\mid \alpha_{12},\alpha_{13},\alpha_{23}\in\mathbb{R}\right\}$. The \sgp\ $S_{111}$ is in fact the group of all real
upper unitriangular $3\times 3$-matrices. The latter group is known to be nilpotent of class~2, and as observed by Mal'cev~\cite{Ma53},
every nilpotent group of class~2 satisfies the \sgp\ identity
\begin{equation}
\label{eq:class2} xzytyzx\bumpeq yzxtxzy.
\end{equation}
Now we verify that~\eqref{eq:Z4} follows from~\eqref{eq:class2}. For this, we substitute in~\eqref{eq:class2} the letter $x_1$ for $x$, the
letter $x_3$ for $z$, the word $x_1x_2x_1$ for $y$, and the letter $x_4$ for $t$. We then obtain the identity
\begin{multline*}
\underbrace{x_1}_x\underbrace{x_3}_z\underbrace{x_1x_2x_1}_y\underbrace{x_4}_t\underbrace{x_1x_2x_1}_y\underbrace{x_3}_z\underbrace{x_1}_x
\bumpeq\\
\underbrace{x_1x_2x_1}_y\underbrace{x_3}_z\underbrace{x_1}_x\underbrace{x_4}_t\underbrace{x_1}_x\underbrace{x_3}_z\underbrace{x_1x_2x_1}_y.
\end{multline*}
Multiplying this identity through by $x_1x_2$ on the left and by $x_2x_1$ on the right, we get~\eqref{eq:Z4}.
\end{proof}

Recall that a \sgp\ identity $u\bumpeq v$ is said to be \emph{balanced} if for every letter the number of occurrences of this letter is the
same in $u$ and $v$. Clearly, the identity~\eqref{eq:Z4} is balanced.

\begin{lemma}[{\mdseries\cite[Lemma~3.3]{Sa87}}]
\label{lm:sapir} If a \sgp\ variety $\mathbf{V}$ satisfies a nontrivial balanced identity of the form $Z_n\bumpeq v$, then all periodic
members of $\mathbf{V}$ are locally finite.
\end{lemma}

Let $\mathbf{U}$ stand for the \sgp\ variety defined by the identity~\eqref{eq:Z4}. Then Lemma~\ref{lm:sapir} ensures that all periodic
members of $\mathbf{U}$ are locally finite while Lemma~\ref{lm:support} and Proposition~\ref{pr:classes} imply that the \sgp\
$\mathrm{UT}_3(\mathbb{R})$ lies in the Mal'cev product $\mathbf{U}\malcev\var\mathrm{D}_3$. The variety $\var\mathrm{D}_3$ is locally
finite as a variety generated by a finite \sgp\ \cite[Theorem~10.16]{BS81}. We see that the variety $\var\mathrm{UT}_3(\mathbb{R})$
satisfies the condition (i$'$).

It remains to verify that $\var\mathrm{UT}_3(\mathbb{R})$ satisfies the condition (ii) as well. Clearly, the involutory version of the
condition (ii) is stronger than its plain version so that it suffices to show that each Zimin word is an involutory isoterm relative to
$\mathrm{UT}_3(\mathbb{R})$ considered as an \is.

Let $\mathrm{TA}_2^1$ stand for the \is\ formed by the $(0,1)$-matrices
$$\begin{pmatrix} 0 & 0\\ 0 & 0\end{pmatrix},\
\begin{pmatrix} 1 & 0\\ 0 & 0\end{pmatrix},\
\begin{pmatrix} 0 & 1\\ 0 & 0\end{pmatrix},\
\begin{pmatrix} 1 & 0\\ 1 & 0\end{pmatrix},\
\begin{pmatrix} 0 & 1\\ 0 & 1\end{pmatrix},\
\begin{pmatrix} 1 & 0\\ 0 & 1\end{pmatrix}$$
under the usual matrix multiplication and the unary operation that swaps each of the matrices $\left(\begin{smallmatrix} 1 & 0\\ 0 &
0\end{smallmatrix}\right)$ and $\left(\begin{smallmatrix} 0 & 1\\ 0 & 1\end{smallmatrix}\right)$ with the other one and fixes the rest four
matrices. By \cite[Corollaries~2.7 and~2.8]{ADV12} each Zimin word is an involutory isoterm relative to $\mathrm{TA}_2^1$. Now consider the
involution sub\sgp\ $M$ in $\mathrm{UT}_3(\mathbb{R})$ generated by the matrices
$$
e=\begin{pmatrix}
1 & 0 & 0\\
  & 1 & 0\\
  &   & 1
\end{pmatrix},\quad
x=\begin{pmatrix}
1 & 0 & 0\\
  & 0 & 0\\
  &   & 1
\end{pmatrix},\quad
\text{and}\quad y=\begin{pmatrix}
1 & 1 & 0\\
  & 0 & 1\\
  &   & 1
\end{pmatrix}. $$
Clearly, for each matrix $(\mu_{ij})\in M$, one has $\mu_{ij}\ge0$ and $\mu_{11}=\mu_{33}=1$, whence the set $N$ of all matrices
$(\mu_{ij})\in M$ such that $\mu_{13}>0$ forms an ideal in $M$. Clearly, $N$ is closed under the skew transposition. A straightforward
calculation shows that, besides $e$, $x$, and $y$, the only matrices in $M\setminus N$ are $xy=\left(\begin{smallmatrix}
1 & 1 & 0 \\
  & 0 & 0\\
  &   & 1
\end{smallmatrix}\right)$ and
$yx=\left(\begin{smallmatrix}
1 & 0 & 0 \\
  & 0 & 1\\
  &   & 1
\end{smallmatrix}\right)$.
Consider the following bijection between $M\setminus N$ and the set of non-zero matrices in $\mathrm{TA}_2^1$:
$$e\mapsto\left(\begin{smallmatrix}
1 & 0 \\ 0 & 1
\end{smallmatrix}\right),\quad
x\mapsto\left(\begin{smallmatrix} 1 & 0 \\ 1 & 0
\end{smallmatrix}\right),\quad
y\mapsto\left(\begin{smallmatrix} 0 & 1 \\ 0 & 0
\end{smallmatrix}\right),\quad
xy\mapsto\left(\begin{smallmatrix} 0 & 1 \\ 0 & 1
\end{smallmatrix}\right),\quad
yx\mapsto\left(\begin{smallmatrix} 1 & 0 \\ 0 & 0
\end{smallmatrix}\right). $$
Extending this bijection to $M$ by sending all elements from $N$ to $\left(\begin{smallmatrix} 0 & 0 \\ 0 & 0
\end{smallmatrix}\right)$ yields an \is\ homomorphism from $M$ onto $\mathrm{TA}_2^1$. Thus, $\mathrm{TA}_2^1$ as a homomorphic image
of an involution sub\sgp\ in $\mathrm{UT}_3(\mathbb{R})$ satisfies all \is\ identities that hold in $\mathrm{UT}_3(\mathbb{R})$. Therefore
each Zimin word is an involutory isoterm relative to $\mathrm{UT}_3(\mathbb{R})$, as required.
\end{proof}

\section{Concluding remarks and an open question}
\label{sec:open problems}

Here we discuss which conditions of Theorem~\ref{thm:matrix} are essential and which can be relaxed.

It should be clear from the above proof of Theorem~\ref{thm:matrix} that the fact that we have dealt with matrices over the field
$\mathbb{R}$ is not really essential: the proof works for every semigroup of the form $\mathrm{UT}_3(R)$ where $R$ is an arbitrary
associative and commutative ring with 1 such that
\begin{equation}
\label{eq:char} \underbrace{1+1+\dots+1}_{\text{$n$ times}}\ne 0
\end{equation}
for every positive integer $n$. For instance, we can conclude that the \sgp\ $\mathrm{UT}_3(\mathbb{Z})$ of all upper triangular integer
$3\times 3$-matrices whose main diagonal entries are 0s and/or 1s is \nfb\ in both plain and \is\ settings.

On the other hand, we cannot get rid of the restriction imposed on the main diagonal entries: as the example reproduced in the introduction
implies, the \sgp\ $\mathrm{T}_3(\mathbb{Z})$ of all upper triangular integer $3\times 3$-matrices is \fb\ as a plain \sgp\ since it
satisfies only trivial \sgp\ identities. In a similar way one can show that $\mathrm{T}_3(\mathbb{Z})$ is \fb\ when considered as an \is\
with the skew transposition. Indeed, the sub\sgp\ generated in $\mathrm{T}_3(\mathbb{Z})$ by the matrix $\zeta=\left(\begin{smallmatrix}
2 & 0 & 0 \\
  & 1 & 1\\
  &   & 2
\end{smallmatrix}\right)$ and its skew transpose $\zeta^D$ is free, and therefore, considered as an \is, it is isomorphic to the free involution
semigroup on one generator, say $z$. However $\FI(\{z\})$ contains as an involution subsemigroup a free involutory semigroup on countably
many generators, namely, $\FI(Z)$ where
$$Z=\{zz^*z,\ z(z^*)^2z,\ \dots,\ z(z^*)^nz,\ \dots\}.$$
Hence $\mathrm{T}_3(\mathbb{Z})$  satisfies only trivial \is\ identities. Of course, the same conclusions persist if we substitute
$\mathbb{Z}$ by any associative and commutative ring with 1 satisfying~\eqref{eq:char} for every $n$.

We can however slightly weaken the restriction on the main diagonal entries by allowing them to take values in the set $\{0,\pm1\}$. The
proof of Theorem~\ref{thm:matrix} remains valid for the resulting [involution] \sgp\ that we denote by $\mathrm{UT}^\pm_3(\mathbb{R})$.
Indeed, the homomorphism $\alpha\mapsto\Diag(\alpha)$ of Lemma~\ref{lm:support} extends to a homomorphism of
$\mathrm{UT}^\pm_3(\mathbb{R})$ onto its 27-element subsemigroup consisting of diagonal matrices. The sub\sgp\ classes of the kernel of
this homomorphism are precisely the sub\sgps\ $S_{000},\dots,S_{111}$ from the proof of Proposition~\ref{pr:classes}, and therefore, the
variety $\var\mathrm{UT}^\pm_3(\mathbb{R})$ satisfies the condition (i$'$) of the stronger form of Theorem~\ref{thm:main}. Of course, the
variety fulfils also the condition (ii) since (ii) is inherited by supervarieties. In the same fashion, the proof of
Theorem~\ref{thm:matrix} applies, say, to the semigroup of all upper triangular complex $3\times 3$-matrices whose main diagonal entries
come from the set $\{0,1,\xi,\dots,\xi^{n-1}\}$ where $\xi$ is a primitive $n$th root of unity.

The question of whether or not a result similar to Theorem~\ref{thm:matrix} holds true for analogs of the \sgp\ $\mathrm{UT}_3(\mathbb{R})$
in other dimensions is more involved. The variety $\var\mathrm{UT}_2(\mathbb{R})$ fulfils the condition (i$'$) since the condition is
clearly inherited by subvarieties and the injective map $\mathrm{UT}_2(\mathbb{R})\to \mathrm{UT}_3(\mathbb{R})$ defined by
$\left(\begin{smallmatrix}
\mathstrut\alpha_{11} & \alpha_{12}\\
\mathstrut            & \alpha_{22}
\end{smallmatrix}\right)\mapsto\left(\begin{smallmatrix}
\alpha_{11} & 0  & \alpha_{12}\\
  & 0            & 0\\
  &              & \alpha_{22}
\end{smallmatrix}\right)$ is an embedding of [involution] \sgps\ whence
$\var\mathrm{UT}_2(\mathbb{R})\subseteq\var\mathrm{UT}_3(\mathbb{R})$. However, $\var\mathrm{UT}_2(\mathbb{R})$ does not satisfy the
condition (ii) as the following result shows.

\begin{proposition}
\label{pr:2x2} The  \sgp\ $\mathrm{UT}_2(\mathbb{R})$ of all upper triangular real $2\times 2$-matrices whose main diagonal entries are
$0$s and/or $1$s satisfies the identity
\begin{equation}
\label{eq:Z4new} Z_4\bumpeq x_1x_2x_1x_3x_1^2x_2x_4x_2x_1^2x_3x_1x_2x_1.
\end{equation}
\end{proposition}

\begin{proof}
Fix an arbitrary homomorphism $\varphi:\{x_1,x_2,x_3,x_4\}^+\to\mathrm{UT}_2(\mathbb{R})$. For brevity, denote the right hand side
of~\eqref{eq:Z4new} by $w$; we thus have to prove that $Z_4\varphi=w\varphi$. Let
$$x_1\varphi=\left(\begin{smallmatrix}
\mathstrut\alpha_{11} & \alpha_{12}\\
\mathstrut            & \alpha_{22}
\end{smallmatrix}\right),\ \ x_2\varphi=\left(\begin{smallmatrix}
\beta_{11} & \beta_{12}\\
           & \beta_{22}
\end{smallmatrix}\right),\ \ x_3\varphi=\left(\begin{smallmatrix}
\gamma_{11} & \gamma_{12}\\
            & \gamma_{22}
\end{smallmatrix}\right),\ \ x_4\varphi=\left(\begin{smallmatrix}
\delta_{11} & \delta_{12}\\
            & \delta_{22}
\end{smallmatrix}\right),$$
where $\alpha_{11},\alpha_{22},\beta_{11},\beta_{22},\gamma_{11},\gamma_{22},\delta_{11},\delta_{22}\in\{0,1\}$ and
$\alpha_{12},\beta_{12},\gamma_{12},\delta_{12}\in\mathbb{R}$.

If $\alpha_{22}=0$, the fact that $\alpha_{11},\beta_{11},\gamma_{11},\delta_{11}\in\{0,1\}$ readily implies that
$Z_4\varphi=w\varphi=\left(\begin{smallmatrix}
\mathstrut \varepsilon & \varepsilon\alpha_{12}\\
\mathstrut             & 0
\end{smallmatrix}\right)$, where $\varepsilon=\alpha_{11}\beta_{11}\gamma_{11}\delta_{11}$.
Similarly, if $\alpha_{11}=0$, then it is easy to calculate that $Z_4\varphi=w\varphi=\left(\begin{smallmatrix}
\mathstrut 0 & \alpha_{12}\eta\\
\mathstrut   & \eta
\end{smallmatrix}\right)$, where $\eta=\alpha_{22}\beta_{22}\gamma_{22}\delta_{22}$.
We thus may (and will) assume that $\alpha_{11}=\alpha_{22}=1$.

Now, if $\beta_{22}=0$, a straightforward calculation shows that $Z_4\varphi=w\varphi=\left(\begin{smallmatrix}
\mathstrut \varkappa & \varkappa(\alpha_{12}+\beta_{12})\\
\mathstrut             & 0
\end{smallmatrix}\right)$, where $\varkappa=\beta_{11}\gamma_{11}\delta_{11}$. Similarly, if $\beta_{11}=0$, we get
$Z_4\varphi=w\varphi=\left(\begin{smallmatrix}
\mathstrut 0 & (\alpha_{12}+\beta_{12})\lambda\\
\mathstrut   & \lambda
\end{smallmatrix}\right)$, where $\lambda=\beta_{22}\gamma_{22}\delta_{22}$. Thus, we may also assume that $\beta_{11}=\beta_{22}=1$.
Observe that the word $w$ is obtained from the word $Z_4$ by substituting $x_1^2x_2$ for the second occurrence of the factor $x_1x_2x_1$
and $x_2x_1^2$ for the third occurrence of this factor. Therefore $\alpha_{11}=\alpha_{22}=\beta_{11}=\beta_{22}=1$ implies
$x_1x_2x_1\varphi=x_1^2x_2\varphi=x_2x_1^2\varphi=\left(\begin{smallmatrix}
\mathstrut 1 & 2\alpha_{12}+\beta_{12}\\
\mathstrut   & 1
\end{smallmatrix}\right)$ whence $Z_4\varphi=w\varphi$.
\end{proof}

Now let $\mathrm{UT}_n(\mathbb{R})$ stand for the \sgp\ of all upper triangular real $n\times n$-matrices whose main diagonal entries are
$0$s and/or $1$s and assume that $n\ge 4$. Here the behaviour of the [involution] \sgp\ variety generated by $\mathrm{UT}_n(\mathbb{R})$
with respect to the conditions of Theorem~\ref{thm:main} is in a sense opposite. Namely, it is not hard to show (by using an argument
similar to the one utilized in the proof of Theorem~\ref{thm:matrix}) that the variety $\var\mathrm{UT}_n(\mathbb{R})$ with $n\ge 4$
satisfies the condition (ii). On the other hand, the approach used in the proof of Theorem~\ref{thm:matrix} fails to justify that this
variety fulfils (i$'$). In order to explain this claim, suppose for simplicity that $n=4$. Then the homomorphism
$\alpha\mapsto\Diag(\alpha)$ maps $\mathrm{UT}_4(\mathbb{R})$ onto its 16-element subsemigroup consisting of diagonal matrices which all
are idempotent. This induces a partition of $\mathrm{UT}_4(\mathbb{R})$ into 16 sub\sgps, and to mimic the proof of
Theorem~\ref{thm:matrix} one should show that all these sub\sgps\ belong to a variety whose periodic members are locally finite. One of
these 16 sub\sgps\ is nothing but the group of all real upper unitriangular $4\times 4$-matrices. The latter group is known to be nilpotent
of class~3, and one might hope to use the identity
\begin{equation}
\label{eq:class3} xzytyzxsyzxtxzy\bumpeq yzxtxzysxzytyzx,
\end{equation}
proved by Mal'cev~\cite{Ma53} to hold in every nilpotent group of class~3, along the lines of the proof of Proposition~\ref{pr:classes}
where we have invoked Mal'cev's identity holding in each nilpotent group of class~2. However, it is known~\cite[Theorem~2]{Zi80} that the
variety defined by~\eqref{eq:class3} contains infinite finitely generated periodic \sgps. Even though this fact does not yet mean that the
condition  (i$'$) fails in $\var\mathrm{UT}_4(\mathbb{R})$, it demonstrates that the techniques presented in this paper are not powerful
enough to verify whether or not the variety obeys this condition. It seems that this verification constitutes a very difficult task as it
is closely connected with Sapir's longstanding conjecture that for each nilpotent group $G$, periodic members of the \sgp\ variety $\var G$
are locally finite, see \cite[Section~5]{Sa87}.

Back to our discussion, we see that Theorem~\ref{thm:main} cannot be applied to the \sgp\ $\mathrm{UT}_2(\mathbb{R})$ and we are not in a
position to apply it to the \sgps\ $\mathrm{UT}_n(\mathbb{R})$ with $n\ge4$. Of course, this does not indicate that these \sgps\ are
\fb---recall that Theorem~\ref{thm:main} is only a sufficient condition for being \nfb. Presently, we do not know which of the \sgps\
$\mathrm{UT}_n(\mathbb{R})$ with $n\ne3$ possess the finite basis property, and we conclude the paper with explicitly stating this open
question in the anticipation that, over time, looking for an answer might stimulate creating new approaches to the \fbp\ for infinite
[involution] \sgps:

\begin{question}
For which $n\ne3$ is the [involution] \sgp\ $\mathrm{UT}_n(\mathbb{R})$ \fb?
\end{question}

\medskip

\small

\noindent\textbf{Acknowledgements.} The author gratefully acknowledges support from the Presidential Programme ``Leading Scientific Schools
of the Russian Federation'', project no.\ 5161.2014.1, and from the Russian Foundation for Basic Research, project no.\ 14-01-00524. In
particular, the latter project was used to support the author's participation in the International Conference on Semigroups, Algebras and
Operator Theory (ICSAOT-2014) held at the Cochin University of Science and Technology as well as his participation in the International
Conference on Algebra and Discrete Mathematics (ICADM-2014) held at Government College, Kattappana. Also the author expresses cordial
thanks to the organizers of these two conferences, especially to Dr. P. G. Romeo, Convenor of ICSAOT-2014, and Sri. G. N. Prakash, Convenor
of ICADM-2014, for their warm hospitality.


\begin{thebibliography}{99}
\bibitem{ACHLV}
K. Auinger, Yuzhu Chen, Xun Hu, Yanfeng Luo, M. V. Volkov, The finite basis problem for Kauffman monoids, Algebra Universalis, accepted. [A
preprint is available under \url{http://arxiv.org/abs/1405.0783}.]

\bibitem{ADV12}
K. Auinger, I. Dolinka, M. V. Volkov, Matrix identities involving multiplication and transposition, J. European Math. Soc. \textbf{14}
(2012), 937--969.

\bibitem{Au66}
A. K. Austin, A closed set of laws which is not generated by a finite set of laws, Quart. J. Math. Oxford Ser. (2) \textbf{17} (1966),
11--13.

\bibitem{BaOl75}
Yu. A. Bahturin, A. Yu. Ol'shanskii, Identical relations in finite Lie rings, Mat. Sb., N. Ser. \textbf{96(138)} (1975), 543--559 [Russian;
English translation: Mathematics of the USSR-Sbornik \textbf{25} (1975), 507--523].

\bibitem{Bi65}
A. P. Biryukov, On infinite collections of identities in semigroups, Algebra i Logika \textbf{4}, no. 2 (1965), 31--32 [Russian].

\bibitem{BCK04}
V. D. Blondel, J. Cassaigne, J. Karhum\"aki, Freeness of multiplicative matrix semigroups, Problem 10.3 in:  V. D. Blondel and A. Megretski
(eds.), \textsf{Unsolved Problems in Mathematical Systems and Control Theory}, Princeton University Press, 2004, 309--314.

\bibitem{BS81}
S. Burris, H. P. Sankappanavar, \textsf{A Course in Universal Algebra}, Springer-Verlag, 1981.

\bibitem{CP61}
A. H. Clifford,  G. B. Preston, \textsf{The Algebraic Theory of Semigroups. Vol.I}, Amer.\ Math.\ Soc., 1961.

\bibitem{KR79}
K. H. Kim, F. Roush,  The semigroup of adjacency patterns of words, in: \textsf{Algebraic Theory of Semigroups}, Colloq. Math. Soc. J\'anos
Bolyai, \textbf{20}, North-Holland, 1979, 281--297.

\bibitem{Kr73}
R. L. Kruse, Identities satisfied by a finite ring. J. Algebra \textbf{26} (1973), 298--318.

\bibitem{Lv73}
I. V. L'vov,  Varieties of associative rings. I. Algebra i Logika \textbf{12} (1973), 269--297 [Russian; English translation: Algebra and
Logic \textbf{12} (1973), 150--167].

\bibitem{Ma53}
A. I. Mal'cev, Nilpotent semigroups,  Ivanov. Gos. Ped. Inst. U\v{c}. Zap. Fiz.-Mat. Nauki \textbf{4} (1953), 107--111 [Russian].

\bibitem{McK70}
R. N. McKenzie, Equational bases for lattice theories, Math.\ Scand. \textbf{27} (1970), 24--38.

\bibitem{OaPo64}
S. Oates, M. B. Powell, Identical relations in finite groups, J. Algebra \textbf{1} (1964), 11--39.

\bibitem{Pe66}
P. Perkins, \textsf{Decision Problems for Equational Theories of Semigroups and General Algebras}, Ph.D. thesis, Univ.\ of California,
Berkeley, 1966.

\bibitem{Pe69}
P. Perkins, Bases for equational theories of semigroups, J. Algebra \textbf{11} (1969), 298--314.

\bibitem{Sa87}
M. V. Sapir, Problems of Burnside type and the finite basis property in varieties of semigroups, Izv. Akad. Nauk SSSR Ser. Mat. \textbf{51}
(1987), 319--340 [Russian; English translation: Math.\ USSR--Izv. \textbf{30} (1987), 295--314].

\bibitem{Vo01}
M. V. Volkov, The finite basis problem for finite semigroups, Sci.\ Math.\ Jpn. \textbf{53} (2001), 171--199. [A periodically updated
version is avalaible under \url{http://csseminar.kadm.usu.ru/MATHJAP_revisited.pdf}.]

\bibitem{Zi80}
A. I. Zimin, Semigroups that are nilpotent in the sense of Mal'cev, Izv. Vyssh. Uchebn. Zaved. Mat. (1980), no.6, 23--29 [Russian].
\end{thebibliography}
\end{document}